\documentclass[preprint,10pt]{elsarticle}

\usepackage{color,latexsym,amsmath,amssymb,amsfonts,amsthm}

\makeatletter
\def\ps@pprintTitle{%
     \let\@oddhead\@empty
     \let\@evenhead\@empty
     \def\@oddfoot{\reset@font {\footnotesize\itshape to appear in J. Pure Appl. Algebra}\hfil\thepage\hfil%
     \llap{\footnotesize\itshape\today}}
     \let\@evenfoot\@oddfoot}
\makeatother

\newtheorem{lemma}{Lemma}[section]
\newtheorem{theorem}[lemma]{Theorem}
\newtheorem{proposition}[lemma]{Proposition}
\newtheorem{corollary}[lemma]{Corollary}

\theoremstyle{definition}
\newtheorem{definition}[lemma]{Definition}

\theoremstyle{remark}
\newtheorem{example}[lemma]{Example}

\newtheorem{remark}[lemma]{Remark}

\def\SN{\mathbb N}    
\def\SP{\mathbb P}    
\def\SQ{\mathbb Q}    
\def\SZ{\mathbb Z}    

\def\II{\mathfrak I}    

\def\int{\mbox{\rm{Int}}}             
\def\lc{{\rm lc}}                       

\def\hZ{\widehat{\SZ}}

\def\Aa{\mathcal A}         
\def\Cc{\mathcal C}         

\def\Ž{\'e}
\def\{\`e}
\def\ˆ{\`a}
\def\{\`u}
\def\™{\^{o}}
\def\{\^{e}}
\def\ž{\^{u}}
\def\{\c c}

\def\uE{\underline{E}}

\def\ua{\underline{\alpha}}
\def\ub{\underline{\beta}}
\def\ux{\underline{x}}
\def\uy{\underline{y}}
\def\uO{\underline{O}}
\def\uV{\underline{V}}
\def\ug{\underline{g}}
\def\uc{\underline{c}}

\def\uga{\underline{\gamma}}
\def\ufi{\underline{\varphi}}
\def\upsi{\underline{\psi}}

\def\be{\begin{equation}}
\def\ee{\end{equation}}

\def\NC{\normalcolor}

\begin{document}
\begin{frontmatter}

\title{Adelic versions of the Weierstrass approximation theorem}

\author{Jean-Luc Chabert}
\address{LAMFA (UMR-CNRS 7352), Universit\'e de Picardie, 33 rue Saint Leu, 80039 Amiens}
\ead{jean-luc.chabert@u-picardie.fr}

\author{Giulio Peruginelli}
\address{Dipartimento di Matematica, Universit\`a di Padova, Via Trieste 63, 35121, Padova}
\ead{gperugin@math.unipd.it}

\begin{abstract} 
Let $\uE=\prod_{p\in\SP}E_p$ be a compact subset of $\hZ=\prod_{p\in\SP}\SZ_p$ and denote by $\Cc(\uE,\hZ)$ the ring of continuous functions from $\uE$ into $\hZ$. We obtain two kinds of adelic versions of the Weierstrass approximation theorem. Firstly, we prove that the ring 
$\int_\SQ(\uE,\hZ):=\{f(x)\in\SQ[x]\mid f(\uE)\subseteq \hZ\}$
is dense in the product $\prod_{p\in\SP}\Cc(E_p,\SZ_p)\,$ for the uniform convergence topology. We also obtain an analogous statement for general compact subsets of $\hZ$.

Secondly, under the hypothesis that, for each $n\geq 0$, $\#(E_p\pmod{p})>n$ for all but finitely many primes $p$, we prove the existence of regular bases of the $\SZ$-module $\int_\SQ(\underline{E},\hZ)$, and show that, for such a basis $\{f_n\}_{n\geq 0}$, every function $\ufi$ in $\prod_{p\in\SP}\Cc(E_p,\SZ_p)$ may be uniquely written as a series $\sum_{n\geq 0}\uc_n f_n$ where $\uc_n\in\hZ$ and $\lim_{n\to \infty}\uc_n\to 0$. Moreover, we characterize the compact subsets $\uE$ for which the ring $\int_\SQ(\uE,\hZ)$ admits a regular basis in P\'olya's sense by means of an adelic notion of ordering which generalizes Bhargava's $p$-ordering.
\end{abstract}


\end{frontmatter}

\bigskip
\section{Introduction: the classical $p$-adic versions of Weierstrass theorem}

A well known $p$-adic version of the Weierstrass polynomial approximation theorem due to Dieudonn\'e \cite{bib:Dieu} states the following: 

\begin{theorem}\label{th:dieu}
For every compact subset $E$ of $\SQ_p$, the ring of polynomial functions $\SQ_p[x]$ is dense in the ring $\Cc(E,\SQ_p)$ of continuous functions from $E$ into $\SQ_p$ with respect to the uniform convergence topology. 
\end{theorem}

If we restrict to functions with values in $\SZ_p$, we have to consider the subring of polynomial functions whose values on $E$ are in $\SZ_p$, namely the ring $\int(E,\SZ_p)=\{f\in\SQ_p[x]\mid f(E)\subseteq \SZ_p\}$, and then, following Kaplansky~\cite{bib:Kap50}, we have:

\begin{theorem}\label{th:kap}
For every compact subset $E$ of $\SQ_p$, the ring of polynomial functions $\int(E,\SZ_p)$ is dense in $\Cc(E,\SZ_p)$, the ring  of continuous functions from $E$ into $\SZ_p$ with respect to the uniform convergence topology. 
\end{theorem} 

In the particular case where $E$ is equal to $\SZ_p$, Mahler \cite{bib:Mahl} gave an effective approximation theorem in terms of the binomial polynomials $\binom{x}{n}:=\frac{x\,(x-1)\,\cdots\,(x-(n-1))}{n!}$.

\begin{theorem}\label{th:mah}
Every continuous function $\varphi\in\Cc(\SZ_p,\SQ_p)$ may be uniquely written as a series in the binomial polynomials $\binom{x}{n}$ with coefficients $c_n$ in $\SQ_p:$
$$\varphi(x)=\sum_{n\geq 0}c_n\binom{x}{n}, \textrm{ where }  v_p(c_n)\to +\infty \textrm{ and }
\inf_{x\in\SZ_p}v_p(\varphi(x))=\inf_{n\geq 0}v_p(c_n).$$
\end{theorem}

It is straightforward to see that, if $\varphi\in\Cc(\SZ_p,\SZ_p)$, then the coefficients $c_n$'s lie in $\SZ_p$ and the partial sums $\sum_{n\geq 0}^d c_n\binom{x}{n}$ are elements of $\int(\SZ_p,\SZ_p)$, so we find again Kaplansky's Theorem when $E=\SZ_p$. Mahler's result was extended to regular compact subsets of $\SQ_p$ by Amice~\cite{bib:Ami}, and then to every compact subset of $\SQ_p$ by Bhargava and Kedlaya \cite{bib:Bharg2}:

\begin{theorem}\label{th:BK}
For each compact subset $E$ of $\SQ_p$, there exists a sequence $\{a_n\}_{n\geq 0}$ of elements of $E$,  called $p$-ordering of $E$, such that the polynomials 
$ f_n(x)=\prod_{k=0}^{n-1}\frac{x-a_k}{a_n-a_k}$, $n\in\SN$, 
form a basis of the $\SZ_p$-module $\int(E,\SZ_p)$. Then, every $\varphi\in\Cc(E,\SQ_p)$ may be uniquely written as a series in the $f_n$'s with coefficients in $\SQ_p:$
$$\varphi(x)=\sum_{n\geq 0}c_n f_n(x)\textrm{ where } c_n\in\SQ_p\textrm{ and } \lim_{n\to+\infty}v_p(c_n)=+\infty\,.$$
 \end{theorem}
 Moreover, one knows that with the previous notation:
$$ \inf_{x\in E}v_p(\varphi(x))=\inf_{n\geq 0}v_p(c_n)\,.$$

Our aim is to obtain adelic versions of these results. In the next section we recall the restricted product  topology on the ring $\Aa_f(\SQ)$ of finite ad\`eles of $\SQ$, since we are interested in the Banach space of continuous functions $\Cc(E,\Aa_f(\SQ))$ where $E$ denotes a compact subset of $\Aa_f(\SQ)$. The subring of polynomial functions that we will consider first is $$\int_{\SQ}(E,\hZ):=\{f\in\SQ[x]\mid f(E)\subseteq\hZ\}$$ where $\hZ$ denotes the profinite completion of $\SZ$. It is well-known that $\hZ$ is isomorphic to $\prod_{p\in\SP}\SZ_p$ and henceforth we will take the latter ring as our model for $\hZ$. In section~3, we establish the following theorem which is somewhat more general than the statement given in the summary:

\begin{theorem}
 Let $E$ be a compact subset of $\hZ$. The topological closure $($with respect to the uniform convergence topology$)$ of the ring of polynomial functions $\int_{\SQ}(E,\hZ)$ in the ring $\Cc(E,\hZ)$ of continuous functions from $E$ into $\hZ$ is the subring formed by the restrictions to $E$ of the functions belonging to the product  $\prod_{p\in\SP}\Cc(E_p,\SZ_p)$ where $E_p$ denotes the image of $E$ by the projection from $\hZ$ onto $\SZ_p$.
\end{theorem} 
 
This statement is noteworthy in so far as a polynomial with a single variable can simultaneously approach several functions (of one variable but on subsets of distinct $\SZ_p$'s). In order to extend Bhargava and Kedlaya's result, we characterize in section~4 the compact subsets $E$ of $\hZ$ such that the $\SZ$-module $\int_{\SQ}(E,\hZ)$ admits a \emph{regular basis} (i.e., a basis with one polynomial of each degree), and then we show how we can construct such a basis provided it exists.  But this is not enough: in order to obtain an adelic analogue of Theorem~\ref{th:BK}, we have to introduce, in section~5,  polynomials whose coefficients are ad\`eles and to consider bases of the $\hZ$-module $\int_{\SQ}(E,\hZ)\otimes_{\SZ}\hZ $.  These bases may be constructed explicitly by means of the notion of \emph{adelic ordering}, which generalizes Bhargava's $p$-ordering. Finally, in section~6, we use all the previous results to obtain the following analogue of Bhargava and Kedlaya's result: 

\begin{theorem}
Let $\uE=\prod_{p\in\SP}E_p$ be a compact subset of $\hZ$ such that all $E_p$'s are infinite and, for each $n\geq 0$, $\#(E_p\pmod{p})>n$ for all but finitely many $p.$ Then, there exist bases of the $\SZ$-module $\int_\SQ(\underline{E},\hZ)$, and, for such a basis $\{f_n\}_{n\geq 0}$, every $\ufi\in\prod_{p\in\SP}\Cc(E_p,\SZ_p)$ may be uniquely written as a series $\sum_{n\geq 0}\uc_n f_n$ where $\uc_n\in\hZ$ and $\lim_{n\to \infty}\uc_n\to 0$. \end{theorem}


\section{The adelic framework}

Let $\Aa_f(\SQ)$ denote the topological {\em ring of finite ad\`eles} of $\SQ$  (see for example \cite[Chapter 1]{bib:GH}). Recall that, as a subset, $\Aa_f(\SQ)$ is the restricted product of the fields $\SQ_p$ with respect to the rings $\SZ_p,$ as $p$ ranges through the set of primes $\SP$, that is:
$$\Aa_f(\SQ)=\left\{\underline{x}=(x_p)\in\prod_{p\in\SP}\SQ_p \mid  x_p\in\SZ_p \textrm{ for all but finitely many } p\in\SP\right\}.$$
The topology on $\Aa_f(\SQ)$ is the restricted product topology characterized by the fact that the following subsets form a basis of open subsets: $\prod_{p\in\SP}O_p$ where $O_p$ is an open subset of $\SQ_p$ and $O_p=\SZ_p$ for all but finitely many $p$.

Clearly, $\Aa_f(\SQ)$ is locally compact and contains the compact subring $\hZ,$ the profinite completion of $\SZ$ with respect to the ideal topology, that is,
$$\hZ=\varprojlim_{n>1}\SZ/n\SZ\cong\prod_{p\in\SP}\SZ_p\,.$$

Note that the topology induced on $\hZ=\prod_p\SZ_p$ by the restricted product topology of $\Aa_f(\SQ)$ is nothing else than the product topology.

Note also that $\SQ$ may be embedded in $\Aa_f(\SQ)$ in the following way: 
\begin{align}
j:\SQ\to&\Aa_f(\SQ)\label{eq:j}\\
\alpha\mapsto& (\alpha_p)_{p\in\SP} \nonumber
\end{align}
where $\alpha_p=\alpha$, for each $p\in\SP$.

\smallskip

For every compact subset $E$ of $\Aa_f(\SQ)$, we consider the $\SQ$-vector space $\Cc(E,\Aa_f(\SQ))$ of continuous functions from $E$ into $\Aa_f(\SQ)$. 
The topology on $\Cc(E,\Aa_f(\SQ))$ that we will consider is the uniform convergence topology with respect to the restricted product topology on $\Aa_f(\SQ)$. Endowed with this topology, the $\SQ$-vector space  $\Cc(E,\Aa_f(\SQ))$ is a Banach space. 

\smallskip

One might be tempted to find some ring of polynomial functions with rational coefficients which could be dense in $\Cc(E,\Aa_f(\SQ))$. In fact, our first result is negative.

\begin{proposition}
The Banach space $\Cc(E,\Aa_f(\SQ))$ does not contains any dense $\SQ$-subspace of polynomial functions.  
\end{proposition}

\begin{proof}
It is enough to prove that even one component of one function restricted to one variable cannot be approximated. Thus, we will prove that $\SQ[x]$ is not dense in $\Cc(\SZ_q,\SZ_p)$ where $p\not=q\in\SP$. Indeed, the function $\chi:\SZ_q\to\SZ_p$ which is the characteristic function of the subset $q\SZ_q$ is continuous. Assume that $f\in\SQ[x]$ is an approximation of $\chi$ mod $p\SZ_p$, that is, $f(q\SZ_p)\subseteq 1+p\SZ_p$ and $f(\SZ_q\setminus q\SZ_q)\subseteq p\SZ_p$. In particular, $f(0)\in 1+p\SZ_p$ while $f(p^k)\in p\SZ_p$ for all $k\geq 1$. The fact that, for every polynomial $g\in\SQ[x]$, $g(p^k)-g(0)\in p\SZ_p$ for $k$ large enough leads to a contradiction.
\end{proof}

Thus, the topological closure of subrings of polynomial functions in $\SQ[x]$ will be strictly contained in $\Cc(E,\Aa_f(\SQ))$. In the sequel, we will adopt the following notation.

\medskip

\noindent{\bf Notation}. 

\noindent -- For every prime $p\in\SZ$, let $\pi_p$ be the canonical surjection from $\Aa_f(\SQ)$ onto $\SQ_p$. 

\noindent -- For every subset $E$ of $\Aa_f(\SQ)$ and $p\in\SP$, let $E_p=\pi_p(E)$, and denote by $\uE$ the cartesian product of these $E_p$'s:

\centerline{$\uE:=\prod_{p\in\SP}\pi_p(E)=\prod E_p\,.$}

The notation $\uE$ will always denote  a subset of $\Aa_f(\SQ)$ that can be written as a cartesian product $\uE=\prod E_p$.

\noindent -- For every function $\psi$ from $E$ into $\Aa_f(\SQ)$,  we denote by $\psi_p$ its components: 

\centerline{$\psi_p=\pi_p\circ\psi:E\to\SQ_p$.}

 Clearly, if $\psi\in\Cc(E,\Aa_f(\SQ))$, then the components $\psi_p$ are continuous and belong to $\Cc(E,\SQ_p)$.

\medskip

We continue with a lemma which collects various results about compact subsets and continuous functions on rings of  ad\`eles.

\begin{lemma}\label{rem:1}
Let $E$ be a compact subset of $\Aa_f(\SQ)$ and $\psi\in\Cc(E,\Aa_f(\SQ))$. Then:
\begin{itemize}
\item[(1)] there exists $d\in\SN^*$ such that $dE\subseteq\hZ$;

\item[(2)] there exists $d_1\in\SN^*$ such that  $d_1\psi\in\Cc(E,\hZ)$;

\item[(3)] let $d$ and $d_1\in\SN^*$ be as in (1) and (2). Then the function $\varphi$ defined by $\varphi(\ux) =d_1\psi(\frac{1}{d}\ux)$ belongs to $\Cc(dE,\hZ).$
\end{itemize}
\end{lemma}
\begin{proof}
(1) The projections $E_p=\pi_p(E)$ of $E$ are compact subsets of $\SQ_p,$ and hence, $\uE=\prod_{\in\SP}E_p$ is compact. Clearly, we have the containment $E\subseteq \uE=\prod_{\in\SP}E_p$. Moreover, by definition of the topology on $\Aa_f(\SQ)$, we necessarily have that $E_p\subseteq\SZ_p$ for all but finitely many $p$, so (1) now follows immediately.

(2) Since $E$ is compact, every continuous function on $E$ with values in $\Aa_f(\SQ)$ is uniformly continuous. It is clear that the image  $\psi(E)$ is a compact subset of $\Aa_f(\SQ)$. By (1) there exists $d_1\in\SN^*$ such that $d_1\cdot\psi(E)\subseteq\hZ$, that is, $d_1\psi\in\Cc(E,\hZ)$. 

(3) Obvious.
\end{proof}

Now, the question could be: which subring of polynomial functions shall we consider? First, we make the choice of considering polynomials in one variable with coefficients in $\SQ$. We will see why this choice is sufficient at least for the first part of our work. 

\medskip

For $f=\sum_{k=0}^nc_kx^k\in\SQ[x]$ and $\ua=(\alpha_p)\in\Aa_f(\SQ)$, the value of $f$ at $\ua$ is:
$$ f(\ua)=\sum_{k=0}^nc_k\,\ua^k=\left(\sum_{k=0}^nc_k\,\alpha_p^k\right)_{p\in\SP}=(f(\alpha_p))_{p\in\SP}\in\prod_{p\in\SP}\SQ_p.$$
Note that these equalities follow from the structure of $\Aa_f(\SQ)$ as $\SZ$-algebra.  Clearly, we have $f(\ua)\in\Aa_f(\SQ)$. Thus, every $f\in\SQ[x]$ may be considered as a function from $\Aa_f(\SQ)$ into itself defined by:
\begin{align}
\Aa_f(\SQ)\to &\Aa_f(\SQ)\label{eq:deff}\\
\ua=(\alpha_p)_{p\in\SP}\mapsto & f(\ua)=(f(\alpha_p))_{p\in\SP}\nonumber
\end{align}
Are they continuous functions? Fix some $\ua\in\Aa_f(\SQ)$ and some basic open neighborhood $\uO=\prod_{p\in\SP}O_p$ of $0$ in $\Aa_f(\SQ)$. Clearly the components of $f$ (which are all equal to $f$) are continuous functions from $\SQ_p$ into $\SQ_p$, and hence, for every $p\in\SP$, there exists an open neighborhood $V_p$ of 0 in $\SQ_p$ such that $(\beta_p-\alpha_p)\in V_p$ implies $(f(\beta_p)-f(\alpha_p))\in O_p$. Let $\SP_1$ be the subset of $\SP$ formed by the $p$'s such that $f\in\SZ_p[x]$, $O_p=\SZ_p$, and $\alpha_p\in\SZ_p$. For $p\in\SP_1$, we may choose $\SZ_p$ instead of $V_p$, and then, since $\SP\setminus\SP_1$ is finite, we may consider the basic open neighborhood of $0:$ $\uV=\prod_{p\in\SP_1}\SZ_p\times\prod_{p\notin\SP_1}V_p$ which is such that $(\ub-\ua)\in\uV\Rightarrow(f(\ub)-f(\ua))\in\uO.$
Finally, $f$ is a uniformly continuous function from $\Aa_f(\SQ)$ into itself. 

In other words, with the previous convention (\ref{eq:deff}) about the polynomial functions~$f$, for each subset $E$ of $\Aa_f(\SQ)$ we have $\SQ[x]\subset  \Cc(E,\Aa_f(\SQ))$.  Now, we may state our adelic theorems.

\section{First adelic Weierstrass theorems}

In the previous section we remarked that for our considerations without loss of generality we may  restrict our proofs to
compact subsets $E$ of $\hZ$ (Lemma \ref{rem:1}). 

\begin{definition}\label{def:int}
Let $E$ be a subset of $\hZ$. The subring of $\Cc(E,\hZ)$ formed by the rational polynomial functions which are {\em integer valued on the subset $E$}, that is, $\Cc(E,\hZ)\cap\SQ[x]$, is denoted by 
$\int_{\SQ}(E,\hZ)=\{f\in\SQ[x]\mid f(E)\subseteq\hZ\}.$ More precisely:
$$\int_{\SQ}(E,\hZ):=\{f\in\SQ[x]\mid \forall \ua=(\alpha_p)_p\in E,\;\;\forall p\in\SP,\;\;f(\alpha_p)\in\SZ_p\}.$$
\end{definition}

This notation is a particular case of the following that we will use in the whole paper: if $\Aa$ is a commutative ring, $A$ and $B$ are subrings of $\Aa$ and $E$ is a subset of $\Aa$, then
$$ \int_{A}(E,B):=\{f\in A[x] \mid f(E)\subseteq B\}\,.$$
When the ring $B$ is an integral domain and $A$ denotes its quotient field, we will omit the index $A$, for example as with the previous notation $\int(E,\SZ_p)$ where $\SQ_p$ is omitted. If $E=B$, then we set $\int_{A}(B,B)=\int_{A}(B)$.

\begin{remark}\label{rem:compact subsets hZ}
Let $E$ be a subset of $\hZ$. It is obvious from Definition~\ref{def:int} (see also \cite{bib:CP}) that 
$$\int_{\SQ}(E,\hZ)=\int_{\SQ}(\uE,\hZ), \quad  \textrm{ where }\uE=\prod_{p\in\SP}E_p\,.$$
\end{remark}

\noindent{\bf Notation}. 

\noindent -- Let $\rho_E$ be the canonical homomorphism from $\Cc(\uE,\hZ)$ to $\Cc(E,\hZ)$, which maps a continuous function $\varphi$ on $\uE$ with values in $\hZ$ to its restriction $\varphi_{|E}$. 

\noindent -- We reserve the notation $\ufi,\upsi$ for maps of $\Cc(\uE,\hZ)$ and $\psi,\varphi$ for maps of $\Cc(E,\hZ)$.\\

 For any subset $E$ of $\hZ$, note that the direct product $\prod_{p\in\SP}\Cc(E_p,\SZ_p)$ canonically embeds into $\Cc(\uE,\hZ)$: the image of this embedding is formed precisely by those continuous functions $\ufi$ from $\uE$ to $\hZ$ whose component $\varphi_p$ depends only on the "$p$-th variable", for $p\in\SP$. \NC We can now state the adelic analogue of Kaplansky's result (Theorem~\ref{th:kap}), which we restate here for the sake of the reader, in a slightly different version.
\begin{theorem}\label{th:cap2}

Let $E$ be a compact subset of $\hZ$. The topological closure $($with respect to the uniform convergence topology$)$ of the ring of polynomial functions $\int_{\SQ}(E,\hZ)$ in the ring $\Cc(E,\hZ)$ of continuous functions from $E$ into $\hZ$ is equal to $\rho_E(\prod_{p\in\SP}\Cc(E_p,\SZ_p))$, where $E_p=\pi_p(E)$ for each $p\in\SP$. 
\end{theorem}

In particular, in the case $E=\uE=\prod_{p\in\SP}E_p$, the topological closure of $\int_{\SQ}(\uE,\hZ)$ in $\Cc(\uE,\hZ)$ is equal to $\prod_{p\in\SP}\Cc(E_p,\SZ_p)$.
\begin{proof}
Assume that $\varphi=(\varphi_p)_{p\in\SP}\in\Cc(E,\hZ)$ belongs to the topological closure of $\int_{\SQ}(E,\hZ)$. Fix some $p\in\SP$. Then, whatever $k\in\SN$, there exists $f\in\int_{\SQ}(E,\hZ)$ such that, for every $\ua\in E$, $v_p(\varphi_p(\ua)-f_p(\ua))\geq k$ (where $f_p=\pi_p\circ f$). The fact that, by definition,  $f_p(\ua)=f(\alpha_p)$ implies that, for all $\ua$ and $\ub\in E$ such that $\alpha_p=\beta_p$, one has $v_p(\varphi_p(\ua)-\varphi_p(\ub))\geq k$. Since $k$ may be as large as we want, we may conclude that $\alpha_p=\beta_p$ implies $\varphi_p(\ua)=\varphi_p(\ub)$. Thus, $\varphi_p(\ua)$ is a function of $\alpha_p$ only, in other words, $\varphi_p\in\Cc(E_p,\SZ_p)$, where $E_p=\pi_p(E)$. Hence, this shows that $\varphi$ is the restriction to $E$ of the map $\ufi\in \Cc(\uE,\hZ)$ whose components are equal to $\varphi_p$, for $p\in\SP$ (thus, more precisely, $\ufi$ is in $\prod_{p\in\SP}\Cc(E_p,\SZ_p)$). Therefore, the topological closure of $\int_{\SQ}(E,\hZ)$ in $\Cc(E,\hZ)$  is contained in $\rho_E(\prod_{p\in\SP}\Cc(E_p,\SZ_p))$.  

Let us prove the reverse containment.
Let $\varphi=(\varphi_p)_{p\in\SP}$ be any element of $\Cc(E,\hZ)$ such that, for each $p\in\SP$, the component $\varphi_p$ may be identified with an element of $\Cc(E_p,\SZ_p)$ that we still denote by $\varphi_p$ (i.e., $\varphi=\ufi_{|E}$, for some $\ufi\in \prod_p\Cc(E_p,\SZ_p)$).   Let $\uO=\prod_{p\in\SP}O_p$ be any basic open neighborhood of 0 in $\hZ$. If $p\in\SP$ is such that $O_p\not=\SZ_p$ (they are at most finitely many), let $k_p\geq 0$ be such that $p^{k_p}\SZ_p\subseteq O_p$. By Theorem~\ref{th:BK} there exists a polynomial function $f_p\in\int(E_p,\SZ_p)$ such that $v_p(\varphi_p(\alpha_p)-f_p(\alpha_p))\geq  k_p$ for every $\alpha_p\in E_p$. For all these $p$, write $f_p(x)=\sum_rc_{p,r}x^r$ with $c_{p,r}\in\SQ_p$. By an extension of the Chinese remainder theorem,
 for each $r\geq 0,$ there exists $c_r\in\SQ$ such that $v_p(c_r-c_{p,r})\geq \sup_{p\in\SP}k_p$ if $O_p\not=\SZ_p$ and $v_p(c_r)\geq 0$ if $O_p=\SZ_p$. Then, for every $p\in\SP$, the polynomial $f(x)=\sum c_rx^r\in\SQ[x]$ satisfies:
\vskip0.3cm
$v_p\left(\varphi_p(\alpha_p)-f(\alpha_p)\right)\geq$
$$\left\{\begin{array}{ll}\inf\{v_p(\varphi_p(\alpha_p)-f_p(\alpha_p)),v_p(f_p(\alpha_p)-f(\alpha_p))\}\geq k_p & \textrm{if }O_p\not=\SZ_p\\
\inf\{\varphi_p(\alpha_p)),v_p(f(\alpha_p))\}\geq 0 & \textrm{if }O_p=\SZ_p\end{array}\right.$$
Thus, $\varphi(\ua)-f(\ua)\in\uO$ for every $\alpha\in E$. Moreover, $f\in\int_{\SQ}(E,\hZ)=\int_{\SQ}(\uE,\hZ)$.
\end{proof}

\begin{remark}\label{rem:C}
At the end of the proof of Theorem~\ref{th:cap2}, we used a kind of Chinese remainder theorem among the different overrings $\int_{\SQ}(E_p,\SZ_p),$  which in fact are localizations of $\int_{\SQ}(\uE,\hZ)$ with respect to $p$ (see formula (\ref{eq:17}) below), namely:

\smallskip

Fix a finite set of primes $P_0=\{p_i\in\SP\mid 1\leq i\leq r\}$ and a corresponding set of positive integers $\{k_i\in\SN^*\mid 1\leq i\leq r\}$. For any $f_i\in\int_{\SQ}(E_{p_i},\SZ_{p_i})$ for $i=1,\ldots,r$, there exists $f\in\int_{\SQ}(\uE,\hZ)$ such that $f\equiv f_i\pmod{p_i^{k_i}}$ for all $p_i\in P_0$ and $f\in\SZ_{(p)}[x]$ for all $p\in\SP\setminus P_0$. Here, $f\equiv  g\pmod{p^k}$ means that $f-g\in p^k\SZ_{(p)}[x]$.
\end{remark}

Now, we consider the $\SQ$-Banach space $\Cc(E,\Aa_f(\SQ))$ and the subring $\SQ[x]$ of polynomial functions. Similarly to the first part of the proof of Theorem~\ref{th:cap2}, we can see that, if $\psi\in\Cc(E,\Aa_f(\SQ))$ is in the topological closure of $\SQ[x]$, then the component $\psi_p$ is a function of only one variable, namely $\alpha_p$, and hence, that $\psi$ may be considered as belonging  to $\rho_E(\prod_{p\in\SP}\Cc(E_p,\SQ_p))$. In fact, $\psi$ belongs to the intersection $$\Cc(E,\Aa_f(\SQ))\, \cap \, \rho_E(\prod_{p\in\SP}\Cc(E_p,\SQ_p))$$
which could be considered as the `restricted product' of the Banach spaces $\Cc(E_p,\SQ_p)$ with respect to the subrings $\Cc(E_p,\SZ_p)$. Now we can state an adelic analogue of Dieudonn\'e's result (Theorem~\ref{th:dieu}):

\begin{theorem}\label{th:dieu1}
Let $E$ be a compact subset of $\Aa_f(\SQ)$. The topological closure of the ring of polynomial functions $\SQ[x]$ in $\Cc(E,\Aa_f(\SQ))$ is equal to the following intersection:
$$\Cc(E,\Aa_f(\SQ))\, \cap \, \rho_E(\prod_{p\in\SP}\Cc(E_p,\SQ_p))\,.$$
\end{theorem}

\begin{proof}
We have just seen that the topological closure of $\SQ[x]$ in $\Cc(E,\Aa_f(\SQ))$ is contained in $\Cc(E,\Aa_f(\SQ))\cap\rho_E(\prod_{p\in\SP}\Cc(E_p,\SQ_p))$. Conversely, let $\psi(x)$ be in the last intersection and let $\uO=\prod_{p\in\SP}O_p$ be any basic open neighborhood of 0 in $\Aa_f(\SQ)$.
By Lemma~\ref{rem:1}(3), there exist $d$ and $d_1\in\SN^*$ such the function $\varphi$ defined by $\varphi(\ux) =d_1\psi(\frac{1}{d}\ux)$ belongs to $\Cc(dE,\hZ)$ where $dE\subseteq\hZ$. Since $d_1\uO=\prod_pd_1O_p$ contains a basic open neighborhood of 0 in $\hZ$, 
Theorem~\ref{th:cap2} shows that there exists $f\in\SQ[x]$ such that $\varphi(\ua)-f(\ua)\in d_1\uO$ for every $\ua\in dE.$ Let $g\in\SQ[x]$ defined by $g(x)=\frac{1}{d_1}f(dx)$. Then, $\psi(\ub)-g(\ub)=\frac{1}{d_1}(\varphi(d\ub)-f(d\ub))\in\uO$ for every $\ub\in E$. Therefore, $\psi(x)$ is in the the topological closure of $\SQ[x]$ in $\Cc(E,\Aa_f(\SQ))$.
\end{proof}

\section{Bases of the $\SZ$-module $\int_{\SQ}(E,\hZ)$}

In order to obtain bases analogous to those of Theorems~\ref{th:mah} and \ref{th:BK}, if there are some, for the $\SQ$-Banach space $\Cc(E,\Aa_f(\SQ))\cap\prod_{p\in\SP}\Cc(E_p,\SQ_p)$, we are looking now for bases of the $\SZ$-module $\int_{\SQ}(E,\hZ)$.

For every compact subset $E$ of $\hZ$, we have that $\int_{\SQ}(E,\hZ)$ is a $\SZ$-module. Does this $\SZ$-module admits a basis?
As noticed in Remark~\ref{rem:compact subsets hZ}, we may assume for simplicity that the compact subset $E$ is of the form $\uE=\prod_p E_p$, and hence, that (see also \cite[(6.1)]{bib:CP}):
\begin{equation}\label{eq:IntQEZ}
\int_{\SQ}(\uE,\hZ)=\bigcap_{p\in\SP}\,\int_{\SQ}(E_p,\SZ_p),
\end{equation}
where $\int_{\SQ}(E_p,\SZ_p)=\{f\in\SQ[x] \mid f(E_p)\subseteq\SZ_p\}$, for each $p\in\SP$. Then, as in the classical case of integer-valued polynomials in number fields, we consider the $n$-th {\em characteristic $\SZ$-module} $\II_{n,\SQ}(E,\hZ)$ of $\int_{\SQ}(E,\hZ)$ which is the set formed by the leading coefficients (denoted by "$\lc$") of the polynomials of degree $\leq n$ in $\int_{\SQ}(E,\hZ):$
$$ \II_n(E)=\II_{n,\SQ}(E,\hZ)=\{\lc(f)\mid f\in\int_{\SQ}(E,\hZ), \deg(f)\leq n\}\,. $$
Clearly,
$$ \SZ=\II_{0}(E)\subseteq \ldots \subseteq \II_{n}(E) \subseteq \II_{n+1}(E) \subseteq \ldots \subseteq \SQ $$

Recall that a sequence of polynomials is said to be a {\em regular basis} in P\'olya's sense~\cite{bib:polya} if this is a basis with exactly one polynomial of each degree.
In our setting, \cite[Proposition II.1.4]{bib:CC} says:

\begin{lemma}\label{th:lemma}
A sequence $\{f_n\}_{n\geq 0}$ of elements of the $\SZ$-module $\int_{\SQ}(E,\hZ)$ is a basis such that $\deg(f_n)=n \;($and hence a regular basis$),$ if and only if, for each $n$, the leading coefficient of $f_n$ generates the $\SZ$-module $\II_{n}(E)$. In particular, $\int_{\SQ}(E,\hZ)$ admits a regular basis if and only if all the $\II_{n}(E)$'s are principal fractional ideals of $\SZ$. 
\end{lemma}

Note that the $\SZ$-module $\II_{n}(E)$ is a principal fractional ideal of $\SZ$ if and only if it is a finitely generated $\SZ$-module.

Now our next task is to study the characteristic $\SZ$-modules $\II_{n}(E)$. Recall that  the characteristic $\SZ$-modules of $\int_{\SQ}(E_p,\SZ_p)$ are:
$$\II_n(E_p)=\II_{n,\SQ}(E_p,\SZ_p)=\{\lc(f)\mid f\in\int_{\SQ}(E_p,\SZ_p), \deg(f)\leq n\}\,.$$
Note that, for simplicity, in this section we write $\II_n(E)$ instead of $\II_{n,\SQ}(E,\hZ)$ and $\II_n(E_p)$ instead of $\II_{n,\SQ}(E_p,\SZ_p)$, but we have to take care that these subsets depend on the field which contains the coefficients of the polynomials. First, we consider their localization with respect to the multiplicative subset $\SZ\setminus p\SZ$. Given a $\SZ$-module $M$, we denote by $M_{(p)}$ the localization of $M$ at $\SZ\setminus p\SZ$, which is isomorphic to the tensor product $M\otimes_{\SZ}\SZ_{(p)}$.

\begin{proposition}\label{th:localizationp}
For every compact subset $E$ and every $p\in\SP:$
\be\label{eq:17} \int_{\SQ}(E,\hZ) \otimes _\SZ \SZ_{(p)}\cong \int_{\SQ}(E_p,\SZ_p)\,. \ee
In particular,
\be\label{eq:18}   \II_{n}(E) \otimes _\SZ \SZ_{(p)}\cong \II_{n}(E_p)\,, \ee
and hence,
\be\label{eq:19}   \II_{n}(E) = \cap_{p\in\SP} \;\II_{n}(E_p)\,. \ee
\end{proposition}

\begin{proof}
The isomorphism in (\ref{eq:17}) is proved once we show the following equality:
\be\label{eq:loc}(\SZ\setminus p\SZ)^{-1}\int_{\SQ}(E,\hZ)=\int_{\SQ}(E_p,\SZ_p),\ee
since $\int_{\SQ}(E,\hZ) \otimes _\SZ \SZ_{(p)}$ is canonically isomorphic to $(\SZ\setminus p\SZ)^{-1}\int_{\SQ}(E,\hZ)$.
Equality~(\ref{eq:loc}) is essentially already contained in \cite{bib:CP}, but for the sake of the reader we give here a self-contained argument.

Clearly the containment $(\SZ\setminus p\SZ)^{-1}\int_{\SQ}(E,\hZ)\subseteq\int_{\SQ}(E_p,\SZ_p)$ holds because $\int_{\SQ}(E,\hZ)$ and $\SZ_{(p)}$ are contained in $\int_{\SQ}(E_p,\SZ_p)$ (see also (\ref{eq:IntQEZ})). We show now that the other containment holds. 

Consider any $f(x)\in \int_{\SQ}(E_p,\SZ_p)$ and let $d$ be a positive integer such that $df(x)\in\SZ[x]$. Write $d=p^st$ where $p\nmid t$ and consider $g(x)=tf(x)$. Then, $g(x)\in\SZ_{(q)}[x]$ for all $q\not=p$, and hence, $g(x)\in\int_{\SQ}(E_q,\SZ_q)$ for all $q\not=p$. Also, $g(x)$ is in $\int_{\SQ}(E_p,\SZ_p)$, because $t$ is invertible in $\SZ_p$. Finally, $g(x)\in\int_{\SQ}(E,\hZ)$, $t\in\SZ\setminus p\SZ$, and $f(x)=\frac{1}{t}g(x)\in (\SZ\setminus p\SZ)^{-1}\int_{\SQ}(E,\hZ).$ Equality (\ref{eq:loc}) is proved. 
The isomorphism (\ref{eq:18}) and the equality (\ref{eq:19}) are easy consequences. 
\end{proof}

If $\int_{\SQ}(E,\hZ)$ admits a regular basis, then all the $\II_{n}(E)$'s are fractional ideals of $\SZ$. In particular, for every $p\in\SP$, all the $\II_{n}(E_p)$ are fractional ideals of $\SZ$, which is equivalent to the fact that the $E_p$'s are infinite since, for $n$ fixed, $\II_n(E_p)$ is a fractional ideal if and only if $\# E_p>n$ (cf. \cite[Exercise II.3]{bib:CC}). The following example shows that this condition is not sufficient.

\begin{example}
If $\uE=\prod_{p\in\SP}\,p\SZ,$ then the polynomials $\frac{1}{p}X$ ($p\in\SP$) are in $\int_{\SQ}(\uE,\hZ)$. Therefore the $\SZ$-module $\II_{1}(\uE)$, which contains (in fact, is equal to) the non-finitely generated $\SZ$-module $\sum_{p\in\SP}\frac{1}{p}\SZ$, is not a fractional ideal of $\SZ$, and $\int_{\SQ}(\uE,\hZ)$ does not admit a regular basis as a $\SZ$-module.
\end{example}

\begin{proposition}
Assume that, for each $p\in\SP$, $E_p$ is infinite and let $\II_{n}(E_p)=p^{-n_p}\SZ_{(p)}$ where $n_p\geq 0$. Then,
\be\label{eq:20} \II_{n}(E)=\sum_{p\in\SP} \frac{1}{p^{n_p}}\SZ\,.\ee
In particular, $\II_{n}(E)$ is a fractional ideal of $\SZ$ if and only if $\II_{n}(E_p)=\SZ_{(p)}$ for all but finitely many $p$'s. If such a condition holds, then 
$$\II_{n}(E)=\frac{1}{\prod_{p\in\SP}\, p^{n_p}}\;Ê\SZ\,.$$
\end{proposition}

Note that  (\ref{eq:20}) is equivalent to the following equalities:
$$\forall p\in\SP,\quad \quad v_p(\II_n(E))=v_p(\II_n(E_p))$$
where $v_p(\II_n(E))=\inf_{a\in\II_n(E)}v_p(a)$ even if $\II_n(E)$ is not a fractional ideal of $\SZ$.

\begin{proof}
We just have to prove equality (\ref{eq:20}). By hypothesis, there exists a polynomial $f(x)$ in $\int_{\SQ}(E_p,\SZ_p)$ of degree $n$ and leading coefficient $p^{-n_p}$, thus $p^{n_p}f(x)\in \SZ_{(p)}[x]$ (see~\cite[Proposition II.1.7]{bib:CC}). We apply now  Remark   \ref{rem:C}: there exists $g\in \int_{\SQ}(E,\hZ)$ of the same degree as the degree of $f(x)$ such that $g\equiv f\pmod{p}$ and $g\in\SZ_{(q)}$, for every prime $q\not=p$. It is easy to see that this implies that the leading coefficient of $f(x)$ is equal to $p^{-n_p}$, so that the $\SZ$-module $M=\sum_{p\in\SP}p^{-n_p}\SZ$ is contained in $\II_{n}(E)$.
Conversely, for each $p\in\SP$, we have the equalities:

\centerline{$M\otimes_{\SZ}\SZ_{(p)}\cong M_{(p)} =\frac{1}{p^{n_p}}\SZ_{(p)}=\II_{n}(E_p)\cong \II_{n}(E)\otimes_{\SZ}\SZ_{(p)}$,} 

\noindent which implies $M=\II_{n}(E)$. 
\end{proof}

\begin{corollary}\label{characteristic ideal fractional}
For every compact subset $E$ and every integer $n$, $\II_{n}(E)$ is a fractional ideal of $\SZ$ if and only if both following conditions hold:
\begin{enumerate}
\item $\# E_p > n$ for all $p\in\SP$,
\item $\#(\, E_p \pmod{p}) > n$ for all but finitely many $p\in\SP$. 
\end{enumerate} 
\end{corollary}

\begin{proof}
It is enough to show that $\II_{n}(E_p)=\SZ_{(p)}$ if and only if $\#(E_p \pmod{p})>n$.  Let $\#(\,E_p \pmod{p}) =k$ and let $\alpha_0,\ldots,\alpha_{k-1}$ be elements of $E_p$ non-congruent modulo $p$. Assume that $k\leq n$, if $a_0,\ldots,a_{k-1}$ are elements of $\SZ$ such that $a_i\equiv \alpha_i\pmod{p}$, then $f(x)=\frac{1}{p}x^{n-k}(x-a_0)\ldots(x-a_{k-1})\in\int_{\SQ}(E_p,\SZ_p)$ and $\frac{1}{p}\in\II_{n}(E_p)$. Assume now that $k\geq n+1$, then whatever $1\leq l\leq n+1$, $v_p(\prod_{h=0}^{l-1}(\alpha_l-\alpha_h))=0$, which implies that $\alpha_0,\ldots, \alpha_{n+1}$ is the beginning of a $p$-ordering of $E_p$ and, in particular, the fact that $v_p(\prod_{h=0}^n(\alpha_{n+1}-\alpha_h))=0$ means that $\II_{n}(E_p)=\SZ_{(p)}$ (about $p$-orderings and links with characteristic ideals see \cite{bib:Bharg1}). 
\end{proof}

Lemma~\ref{th:lemma} with the previous corollary implies that:

\begin{corollary}\label{th:mod}
For every compact subset $E$ of $\hZ$, the $\SZ$-module $\int_{\SQ}(E,\hZ)$ admits a regular basis if and only if, both following conditions hold:
\begin{enumerate}
\item $\# E_p$ is infinite for all $p\in\SP$,
\item for each n, $\#(\,E_p \pmod{p}) > n$ for all but finitely many $p\in\SP$.
\end{enumerate}  
\end{corollary}

\begin{remark}
We recall that the family of overrings of $\int(\SZ)=\{f\in\SQ[x] \mid f(\SZ)\subseteq\SZ\}$ contained in $\SQ[x]$ is formed exactly by the rings $\int(\uE,\hZ)$, as $\uE$ ranges through the subsets of $\hZ$ of the form $\prod_pE_p$, where for each prime $p$, $E_p$ is a closed subset of $\SZ_p$ (see \cite[Theorem 6.2]{bib:CP}). Among these rings we can find the subfamily of integer-valued polynomials rings over an infinite  subset $F$ of $\SZ$, that is, $\int(F,\SZ)=\{f\in\SQ[x]\mid f(F)\subseteq\SZ\}$. As already recalled just after Proposition~\ref{th:localizationp}, each ring of the last subfamily has a regular basis (if $F$ is an infinite subset of $\SZ$, each of the characteristic ideals of $\int(F,\SZ)$ is a finitely generated $\SZ$-module \cite[Corollary II.1.6]{bib:CC}). We show in Example \ref{overring with regular basis} below that there are rings  $\int(\uE,\hZ)$ which have a regular basis but are not of the form $\int(F,\SZ)$ for any infinite subset $F$ of $\SZ$.
\end{remark}
\begin{example}\label{overring with regular basis}
For each $p\in\SP$, choose $\alpha_p\in\SZ_p\setminus\SZ_{(p)}$ and let $E_p=\{\alpha_p+j+p^n\mid 0\leq j\leq p-1, n\geq 0\}$.
The set $E_p$ is $p$-adicaly closed since it is the union of $p$ convergent sequences. Moreover, $\#(E_p \pmod{p})=p,$ and hence, $\int_{\SQ}(\prod_pE_p,\hZ)$ has regular bases although it is not of the form $\int(F,\SZ)$ for some $F\subseteq \SZ:$ in fact, $E_p\cap\SZ=\emptyset$ for each $p\in\SP$, so that $\bigcap_{p\in\SP}(E_p\cap\SZ)$ is not dense in $E_p$ for each $p\in\SP$. By \cite[Corollary 6.9]{bib:CP}, the ring $\int_{\SQ}(\prod_pE_p,\hZ)$ cannot be represented as $\int(F,\SZ)$ for any $F\subseteq \SZ$.
\end{example}

If the $\SZ$-module $\int_{\SQ}(E,\hZ)$ admits a regular basis $\{f_n\}_{n\geq 0}$, then, clearly, such a basis is also a regular basis of the $\SZ_{(p)}$-module $\int_{\SQ}(E_p,\SZ_p)$ for every $p\in\SP$, by (\ref{eq:17}). Conversely, from regular bases $\{f_{n,p}\}_{n\geq 0}$ of $\int_{\SQ}(E_p,\SZ_p)$ for every $p\in\SP$, one may construct a regular basis of $\int_{\SQ}(E,\hZ).$ To show how we can do this we first recall how one can construct a basis of the $\SZ_p$-module $\int(E_p,\SZ_p)=\int_{\SQ_p}(E_p,\SZ_p)=\{f\in\SQ_p[x] \mid f(E_p)\subseteq\SZ_p\}$ by means of the notion of Bhargava's $p$-ordering. From this basis we will deduce a basis for the $\SZ_{(p)}$-module $\int_{\SQ}(E_p,\SZ_p)=\int(E_p,\SZ_p)\cap\SQ[x]$ using the fact that
\begin{equation}\label{eq:completion}
\int_{\SQ}(E_p,\SZ_p)\otimes_{\SZ}\SZ_p\cong\int(E_p,\SZ_p)
\end{equation}

Recall that a $p$-{\em ordering} of a subset $E_p$ of $\SZ_p$ is a sequence $\{a_{n}\}_{n\geq 0}$ of elements of $E_p$ such that:
$$\forall n\geq 1,\quad  v_p(\prod_{k=0}^{n-1}(a_n-a_k))=\min_{x\in E_p} v_p(\prod_{k=0}^{n-1}(x-a_k))\,.$$

Such a sequence always exists, and its elements are distinct if and only if $E_p$ is infinite. Then, it is quite obvious that the polynomials
$$ f_{0,p}=1\quad\textrm{and,\;for }n\geq 1, \;f_{n,p}(x)=\prod_{k=0}^{n-1}\frac{x-a_k}{a_n-a_k}  $$
form a regular basis of the $\SZ_p$-module $\int(E_p,\SZ_p)$ (cf. Bhargava~\cite{bib:bhargava2000}). 

For each $n\in\SN$, we set
$$g_{n,p}(x)=\prod_{k=0}^{n-1}(x-a_k),\;\;w_p(n)=v_p(\prod_{k=0}^{n-1}(a_n-a_k)).$$
It is clear that $g_{n,p}(E_p)\subseteq p^{w_p(n)}\SZ_p$. Let $h_{n,p}\in \SZ[x]$ be of degree $n$ such that $h_{n,p}\equiv g_{n,p}\pmod{p^{w_p(n)}}$. It is immediate to see that $$\frac{h_{n,p}(x)}{p^{w_p(n)}}\in \int_{\SQ}(E_p,\SZ_p).$$
Therefore the polynomials $h_{n,p}(x)/p^{w_p(n)}$, $n\in\SN$, form a basis of $\int_{\SQ}(E_p,\SZ_p)$.

\medskip

\noindent{\bf Construction of a $\SZ$-basis of $\int_{\SQ}(E,\hZ)$}.

Suppose that for each $p\in\SP$, $E_p$ is an infinite subset of $\SZ_p$ and let $\{f_{n,p}\}_{n\in\SN}$ be a regular basis of $\int_{\SQ}(E_p,\SZ_p)$. For a fixed $n$, let $P_n$ be the set of primes $p$ such that $\#(E\pmod{p})\leq n$. By Corollary~\ref{th:mod}, $P_n$ is finite. By the Chinese remainder theorem (cf. Remark~\ref{rem:C}), there exists $f_n\in\int_{\SQ}(E,\hZ)$ such that $f_n\equiv f_{n,p} \pmod{p}$ for every $p\in P_n$ and $f_n\in\SZ_{(q)}[x]$ for every $q\in\SP\setminus P_n$. Let $c_n$ and $c_{n,p}$ denote the leading coefficients of $f_n$ and $f_{n,p}$, respectively. Then, for $p\in P_n$, $v_p(c_{n,p})<0$ and $v_p(c_n-c_{n,p})>0$ implies $v_p(c_n)=v_p(c_{n,p})=v_p(\II_n(E_p))=v_p(\II_n(E))$. Write $c_n=\frac{a}{b}$ where $(a,b)=1$. Consequently, for every $p\in P_n$, $v_p(c_n)=-v_p(b)=v_p(\II_n(E))$ and, for every $q\in\SP\setminus P_n$,   $v_q(c_n)=v_q(a)\geq 0$ . Let $u,v\in\SZ$ be such that $ua+vb=1$ and consider the polynomial $g_n=uf_n+vx^n\in\int_{\SQ}(E,\hZ)$ with $\deg(g_n)=n$ and leading coefficient $\frac{1}{b}$. Finally, by Lemma~\ref{th:lemma}, since $v_p(\frac{1}{b})=v_p(\II_n(E))$ for all $p\in\SP$, the $g_n$'s form a regular basis of $\int_{\SQ}(E,\hZ)$.

\medskip

The next remark explains why, finally, the ring $\int_{\SQ}(E,\hZ)$ is not enough for the purpose of writing any element of the $\SQ$-Banach space
$\prod_{p\in\SP}\Cc(E_p,\SQ_p)$ as a sum of a series in the elements of a regular basis of $\int_{\SQ}(E,\hZ)$ with integer coefficients.

\begin{remark}
Assume that the $\SZ$-module $\int_{\SQ}(E,\hZ)$ admits a regular basis $\{f_n\}_{n\geq 0}$. Although, by Theorem~\ref{th:cap2}, the $\SZ$-linear combinations of the $f_n$'s may uniformly approximate every element $\underline{\varphi}=(\varphi_p)_{p\in\SP}$ of the ring $\prod_{p\in\SP}\Cc(E_p,\SZ_p),$
these $\SZ$-linear combinations $\sum_{n=0}^{d}c_n f_n$ are not enough to allow us to write such a $\underline{\varphi}$ as the sum of a series in the $f_n$'s with coefficients in $\SZ$. The reason is that the coefficients $c_n$'s of the $f_n$'s strongly depend on the neighborhood $O$ chosen for the approximation of $\underline{\varphi}:$ it follows from Theorem~\ref{th:BK} that, for a decreasing sequence of neighborhoods of $0$ of the form $O_l=p_0^l\SZ_{p_0}\times\prod_{p\not=p_0}\SZ_p$, the coefficients $c_n$ are uniquely determined by the component $\varphi_{p_0}$ of $\underline{\varphi}$. And, generally, they are different for another component $\varphi_p$ since the components of $\underline{\varphi}$ may be chosen independently of each other. Thus, we will have to consider coefficients in $\Aa_f(\SQ)$ instead of $\SQ$ and this leads to the next section.
\end{remark}

\section{Polynomials with adelic coefficients}

Let us consider now polynomials $\ug(x)$ with coefficients in $\Aa_f(\SQ):$
$$\ug(x)=\sum_{k=0}^n \underline{\gamma}_kx^k, \textrm{ where } \underline{\gamma}_k=(\gamma_{k,p})_{p\in\SP}\in\Aa_f(\SQ).$$
As before, let $\pi_p$ be the canonical projection from $\Aa_f(\SQ)[x]$ to $\SQ_p[x]$. For a polynomial $\ug(x)$ in $\Aa_f(\SQ)[x]$ we consider its components $\pi_p(\ug(x))=g_p(x)\in \SQ_p[x]$, so that we have:
\begin{equation}\label{eq:adelic polynomial}
\ug=(g_p)_{p\in\SP}, \textrm{ where } g_p(x)=\sum_{k=0}^n \gamma_{k,p}x^kÊ\textrm{ with } \gamma_{k,p}\in\SQ_p\,
\end{equation}
corresponding to the containment $\Aa_f(\SQ)[x]\subset \prod_{p\in\SP}\SQ_p[x]$. Note that the last containment is strict, even if we consider the restricted product of the $\SQ_p[x]$'s with respect to the $\SZ_p[x]$'s because of the degrees of the $g_p$'s which are bounded: for every $\ug\in\Aa_f(\SQ)[x]$, $\deg(g_p)\leq \deg(\ug)$ for all $p$.
\smallskip

Similarly to the case of $\SQ[x]$, for $\underline{g}\in\Aa_f(\SQ)[x]$ and $\ua\in \Aa_f(\SQ)$, we may consider the value of $\underline{g}(x)$ at $\ua:$
$$\underline{g}(\ua)=\sum_{k=0}^n \uga_k\ua^k=\left(\sum_{k=0}^n\gamma_{k,p}\alpha_p^k\right)_{p\in\SP}=(g_p(\alpha_p))_{p\in\SP}\,.$$
Thus, every  $\underline{g}\in\Aa_f(\SQ)[x]$ may be considered as a function from $\Aa_f(\SQ)$ into itself:
\begin{align*}
\Aa_f(\SQ)\to &\Aa_f(\SQ)\\
\ua=(\alpha_p)_{p\in\SP}\mapsto & \ug(\ua)=(g_p(\alpha_p))_{p\in\SP}
\end{align*}
and, analogously to the containment $\SQ[x]\subset\Cc(E,\Aa_f(\SQ)),$ we have the  containment $\Aa_f(\SQ)[x]\subset \Cc(E,\Aa_f(\SQ))$ for any subset $E$ of $\Aa_f(\SQ)$.

\smallskip

From now on, for simplicity, we will consider only compact subsets of $\Aa_f(\SQ)$ of the form $\uE=\prod_{p\in\SP}E_p$.
Then, as for $\SQ[x]$, we also have the following containment: 
$$\Aa_f(\SQ)[x] \subset \prod_{p\in\SP}\Cc(E_p,\SQ_p)\,.$$
Indeed, let $\ug=(g_p)_{p\in\SP}\in\Aa_f(\SQ)[x]$ and fix some $p\in\SP$. The fact that $(\ug(\ua))_p=g_p(\alpha_p)$ means that the value of the $p$-component $g_p(x)$ of $\ug(x)$ at $\ua$ depends only on $\alpha_p$, that is, $g_p\in \Cc(E_p,\SQ_p)$.
 Putting together all these containments, we obtain:
\be\label{eq:cont1} \SQ[x]\subset \Aa_f(\SQ)[x] \subset \Cc(\uE,\Aa_f(\SQ))\cap\prod_{p\in\SP}\Cc(E_p,\SQ_p)\subset \Cc(\uE,\Aa_f(\SQ))\,,\ee
which shows in particular that the topological closure of $\Aa_f(\SQ)[x]$ in $ \Cc(\uE,\Aa_f(\SQ))$ with respect to the uniform convergence topology is equal to topological closure of $\SQ[x]$, namely, $\Cc(\uE,\Aa_f(\SQ))\cap\prod_{p\in\SP}\Cc(E_p,\SQ_p)$ (Theorem \ref{th:dieu1}).

\medskip

Considering now polynomial functions with values in $\hZ$ instead of $\Aa_f(\SQ)$ (note that $\Aa_f(\SQ)$ is equal to $\SQ\otimes_{\SZ}\hZ$\,, the localization of $\hZ$ at $\SZ\setminus\{0\}$), we are led to introduce the following polynomial ring of integer-valued polynomials on $\uE=\prod_pE_p:$
$$\int_{\Aa_f(\SQ)}(\uE,\hZ):=\{\ug\in\Aa_f(\SQ)[x]\mid \ug(\uE)\subseteq \hZ\}\,.$$
Clearly, by definition, for any $\ug=(g_p)_p\in\Aa_f(\SQ)[x]$, where the $g_p$'s are the components of $\ug$ as in (\ref{eq:adelic polynomial}), we have $\ug\in\int_{\Aa_f(\SQ)}(\uE,\hZ)$ if and only if $g_p(E_p)\subseteq\SZ_p$, for each $p\in\SP.$ Consequently, for every $\ug=(g_p)_p\in\int_{\Aa_f(\SQ)}(\uE,\hZ)$, for all but finitely many $p\in\SP$, $g_p\in\SZ_p[x]$ and, for the other $p$, $g_p\in\int_{\SQ_p}(E_p,\SZ_p)$, in other words, $\ug$ belongs to the restricted product of the rings $\int_{\SQ_p}(E_p,\SZ_p)$ with respect to the subrings $\SZ_p[x]$. Note that
$$\int_{\Aa_f(\SQ)}(\uE,\hZ)\cap\SQ[x]=\int_{\SQ}(\uE,\hZ).$$

By intersection with $\Cc(\uE,\hZ),$ containments (\ref{eq:cont1}) lead to:
$$\int_{\SQ}(\uE,\hZ)\subset \int_{\Aa_f(\SQ)}(\uE,\hZ) \subset  \prod_{p\in\SP}\Cc(E_p,\SZ_p) \subset \Cc(\uE,\hZ)\,.$$
\smallskip
Once again, in order to study $\int_{\Aa_f(\SQ)}(\uE,\hZ)$, we introduce its characteristic modules:
$$\II_{n,\Aa_f}(\uE)= \II_{n,\Aa_f(\SQ)}(\uE,\hZ):=\{\lc(\ug)\mid \ug\in\int_{\Aa_f(\SQ)}(\uE,\hZ), \deg(\ug)\leq n\}\,. $$
This is a sub-$\SZ$-module of $\Aa_f(\SQ)$. It is easy  to see that
$$\pi_p(\II_{n,\Aa_f}(\uE))=\II_{n,\SQ_p}(E_p)$$
where
$$\II_{n,\SQ_p}(E_p):=\{\lc(f)\mid f\in\int_{\SQ_p}(E_p,\SZ_p), \deg(f)\leq n\}\,.$$

Indeed, if $\ug\in \int_{\Aa_f(\SQ)}(\uE,\hZ)$ then, for every $p\in\SP$, $g_p\in\int_{\SQ_p}(E_p,\SZ_p)$, and hence, 
$\pi_p(\II_{n,\Aa_f}(\uE))\subseteq \II_{n,\SQ_p}(E_p)$. 
Conversely, if $c\in\II_{n,\SQ_p}(E_p)$, there exists $f\in\int_{\SQ_p}(E_p,\SZ_p)$ with leading coefficient $c$. 
Let $\ug\in\int_{\Aa_f(\SQ)}(\uE,\hZ)$ be such that $g_p=f$ and $g_q=0$ for all $q\not= p$, then $\pi_p(\lc(\ug))=c$.
 
Recall that:
$$\II_n(E_p)=\{\lc(f)\mid f\in\int_{\SQ}(E_p,\SZ_p), \deg(f)\leq n\}\,,$$
$$\II_n(\uE)=\{\lc(\ug)\mid \ug\in\int_{\SQ}(\uE,\hZ), \deg(\ug)\leq n\}\,.$$
We saw in Proposition \ref{th:localizationp} that
$$\II_n(\uE)\otimes_{\SZ}\SZ_{(p)} \cong \II_n(E_p)$$
and it is straightforward that (see (\ref{eq:completion}))
$$\II_n(E_p)\otimes_{\SZ}\SZ_p  \cong \II_{n,\SQ_p}(E_p)\,.$$
The next proposition gives the link between $\int_{\Aa_f(\SQ)}(\uE,\hZ)$ and $\int_{\SQ}(\uE,\hZ):$ 

\begin{theorem}\label{th:52}
For every $\uE=\prod_{p\in\SP}E_p\subseteq\Aa_f(\SQ)$, we have
$$\int_{\Aa_f(\SQ)}(\uE,\hZ) \cong \int_{\SQ}(\uE,\hZ)\otimes_{\SZ}\hZ\,.$$
In particular, 
$$\II_{n,\Aa_f}(\uE) \cong \II_n(\uE)\otimes_{\SZ}\hZ,\;\forall n\geq 0,$$
and if $\{f_n\}_{n\in\SN}$ is a regular basis of the $\SZ$-module $\int_{\SQ}(\uE,\hZ)$, then $\{f_n\}_{n\in\SN}$ is a regular basis of the $\hZ$-module $\int_{\Aa_f(\SQ)}(\uE,\hZ)$.
\end{theorem}

\begin{proof}
Let $\eta:\SQ\otimes_{\SZ}\hZ\cong \Aa_f(\SQ)$ be the canonical isomorphism of $\hZ$-modules characterized by $\eta(r\otimes \underline{1})=j(r)=r\underline{1}$ where $j$ denotes the embedding $\SQ\to \Aa_f(\SQ)$ defined by (\ref{eq:j}). Then, $\eta$ induces another canonical isomorphism of $\hZ$-modules $\widetilde{\eta}:\SQ[x]\otimes_{\SZ}\hZ\cong \Aa_f(\SQ)[x],$ itself characterized by:
 \be\label{eq:etatilde} \forall \, g(x)=\sum_{n=0}^dc_nx^n\in\SQ[x]\;,\quad\widetilde{\eta}(g(x)\otimes \underline{1})=\sum_{n=0}^d j(c_n)\,x^n = \underline{1} \, g(x)\,.\ee
With these identifications, we have the following containment $\int_{\SQ}(\uE,\hZ)\otimes_{\SZ}\hZ\subseteq \Aa_f(\SQ)[x]$ since $\hZ$ is a faithfully flat $\SZ$-module 
(see for instance \cite[Chap. I, \S. 3, 1. Exemples]{bib:Bourb}). In fact, by formulas (\ref{eq:deff}) and (\ref{eq:etatilde}), it is clear that the evaluation of any polynomial in $\int_{\SQ}(\uE,\hZ)\otimes_{\SZ}\hZ$ at any element of $\uE$ is in $\hZ$, so it follows that $\int_{\SQ}(\uE,\hZ)\otimes_{\SZ}\hZ$ may be considered a subring of $\int_{\Aa_f(\SQ)}(\uE,\hZ)$. Conversely, let 
$$\ug(x)=(g_p(x))=\sum_{n=0}^d\gamma_{n,p}x^n\,$$ 
be any element of $\int_{\Aa_f(\SQ)}(\uE,\hZ)$. We have to show that $\ug(x)$ can be written as a finite linear combination of elements of $\int_{\SQ}(\uE,\hZ)$ with coefficients in $\hZ$. There is a finite set $P_0=\{p_i\mid i\in I\}$ of primes such that $g_p\in\SZ_p[x]$ for $p\notin P_0,$ and we know that, for $p\in P_0$, $g_p\in\int_{\SQ_p}(E_p,\SZ_p)$. For $0\leq n\leq d$, we define $\underline{\delta}_{n}\in \hZ$ by $\delta_{n,p}=\gamma_{n,p}$ for $p\notin P_0$ and $\delta_{n,p}=0$ for $p\in P_0$. Since, for each $p$, $\int_{\SQ_p}(E_p,\SZ_p)\cong\int_{\SQ}(\uE,\hZ)\otimes_{\SZ}\SZ_p$ (by (\ref{eq:17}) and (\ref{eq:completion})), there exists a finite set of polynomials $\{h_j\mid j\in J\}\subset \int_{\SQ}(\uE,\hZ)$ such that, for $p\in P_0,$ we may write $g_p(x)=\sum_{j\in J}\varepsilon_{j,p}h_j(x)$ where $\varepsilon_{j,p}\in \SZ_p$. Now, for $0\leq i\leq s$, let $\underline{\nu}_i=(\nu_{i,p})\in\hZ$ with $\nu_{i,p_i}=1$ and $\nu_{i,p}=0$ for $p\not= p_i$. Finally, we have:
$$\ug(x)=\sum_{n=0}^d\underline{\delta}_nx^n +\sum_{j\in J}\left(\sum_{i\in I} \underline{\nu}_i\varepsilon_{j,p_i}\right)h_j(x)\,,$$
which shows that $\ug(x)$ is a finite linear combination of elements of $\int_{\SQ}(\uE,\hZ)$ with coefficients in $\hZ$.
\end{proof}

In this adelic framework, Lemma~\ref{th:lemma} becomes:

\begin{lemma}\label{th:lemmabis}
A sequence $\{f_n\}_{n\geq 0}$ of elements of the $\hZ$-module $\int_{\Aa_f(\SQ)}(\uE,\hZ)$ is a regular basis if and only if, for each $n$, the leading coefficient of $f_n$ generates the $\hZ$-module $\II_{n,\Aa_f}(\uE)$. In particular, $\int_{\Aa_f(\SQ)}(\uE,\hZ)$ admits a regular basis if and only if the $\II_{n,\Aa_f}(\uE)$'s are cyclic $\hZ$-modules. 
\end{lemma}

\begin{proposition}\label{th:cyclic characteristic ideals}
The $\hZ$-module $\II_{n,\Aa_f}(\uE)$ is cyclic if and only if the $\SZ$-module $\II_n(\uE)$ is cyclic.
\end{proposition}

\begin{proof}
It is clear that, if $\II_n(\uE)$ is cyclic, then $\II_{n,\Aa_f}(\uE)=\II_n(\uE)\otimes_{\SZ}\hZ$ is also cyclic. Conversely, if $\II_{n,\Aa_f}(\uE)=\ua\hZ$ where $\ua\in\Aa_f(\SQ)=\hZ\otimes_{\SZ}\SQ$, then there exists $d\in\SZ\setminus\{0\}$ such that $d\ua\in\hZ$, and hence, $d\II_n(\uE)\subseteq\SZ$.
\end{proof}

It follows from Lemma \ref{th:lemma}, Lemma \ref{th:lemmabis} and Proposition \ref{th:cyclic characteristic ideals} that if the $\hZ$-module $\int_{\Aa_f(\SQ)}(\uE,\hZ)$ has a regular basis then the $\SZ$-module $\int_{\SQ}(\uE,\hZ)$ has a regular basis.

Recall that, following Corollary \ref{th:mod}, all the characteristic modules are cyclic if and only if
\begin{enumerate}
\item $\# E_p$ is infinite for all $p\in\SP$,
\item for each $n\in\SN$, $\#(\,E_p \pmod{p}) > n$ for all but finitely many $p\in\SP$.
\end{enumerate}  

Assuming that these conditions are satisfied, we show now how we can construct regular bases of the $\hZ$-module $\int_{\Aa_f(\SQ)}(\uE,\hZ)$. We first extend Bhargava's notion of $p$-ordering.

\begin{definition}\label{def:adel}
An {\em adelic ordering} of $\uE$ is a sequence $\{\ua_n\}_{n\geq 0}$ of elements of $\uE$ such that both following conditions hold:
\begin{enumerate}
\item[(a)] $\forall p\in\SP,\;\forall n\geq 1,\; v_p\left(\pi_p\left(\prod_{k=0}^{n-1}(\ua_n-\ua_k)\right)\right) = \min_{\uy\in\uE}\left\{v_p\left(\pi_p\left(\prod_{k=0}^{n-1}(\uy-\ua_k)\right)\right)\right\} $
\item[(b)] $\forall n\geq 1,\;\;v_p\left(\pi_p\left(\prod_{k=0}^{n-1}(\ua_n-\ua_k)\right)\right)=0$ for almost $p\in\SP$.
\end{enumerate}
\end{definition}

The first condition means that, for each $p\in\SP$, the sequence $\{\alpha_{n,p}\}$ is a $p$-ordering of $E_p$, while the second condition means that the product $\prod_{k=0}^{n-1}(\ua_n-\ua_k)$ is invertible in $\Aa_f(\SQ)$. Since $v_p(\pi_p(\prod_{k=0}^{n-1}(\ua_n-\ua_k)))=v_p(\prod_{k=0}^{n-1}(\alpha_{n,p}-\alpha_{k,p}))$ which is the value of the $p$-sequence of $E_p$ at $n$, condition b) above is easily seen to be equivalent to condition (2) of Corollary \ref{characteristic ideal fractional}.

\begin{proposition}\label{th:56}
Let $\{\ua_n\}_{n\geq 0}$ be a sequence of elements of $\uE$ such that, for each $p$, the $\alpha_{n,p}$'s are distinct. Consider the associated polynomials:
$$\ug_0=1\textrm{ and, for } n\geq 1,\;\;\ug_n(x)=\prod_{k=0}^{n-1}\frac{x-\ua_k}{\ua_n-\ua_k}\,.$$
The following assertions are equivalent:
\begin{enumerate}
\item The sequence $\{\ua_n\}_{n\geq 0}$ is an adelic ordering of $\uE$.
\item The polynomials $\ug_n$ belong to $\int_{\Aa_f(\SQ)}(\uE,\hZ).$
\item The polynomials $\ug_n$ form a regular basis of the $\hZ$-module $\int_{\Aa_f(\SQ)}(\uE,\hZ).$ 
\item For every polynomial $\ug\in\Aa_f(\SQ)[x]$ of degree $n$, 
$$\ug\in\int_{\Aa_f(\SQ)}(\uE,\hZ)\;\Leftrightarrow\; \ug(\ua_k)\in\hZ\textrm{ for } 0\leq k\leq n\,.$$
\end{enumerate}
\end{proposition}

\begin{proof}
Let us prove first that (1)$\leftrightarrow$(2): the coefficients of the polynomials $\ug_n$ belong a priori to $\prod_{p\in\SP}\SQ_p$, to say that there are in $\Aa_f(\SQ)$ is equivalent to the second condition of Definition~\ref{def:adel}. Moreover, $\ug_n\in\int_{\Aa_f(\SQ)}(\uE,\hZ)$ is equivalent to: 
\begin{center}
for every $p\in\SP$, $g_{n,p}(x)=\prod_{k=0}^{n-1}\frac{x-\alpha_{k,p}}{\alpha_{n,p}-\alpha_{k,p}}\in\int_{\SQ_p}(E_p,\SZ_p)$,
\end{center}
which is known to be equivalent to the fact that the sequence $\{\alpha_{n,p}\}_{n\geq 0}$ is a $p$-ordering of $E_p$ (cf. \cite{bib:Bharg1}), which is the first condition of Definition~\ref{def:adel}.

Clearly, (3)$\rightarrow$(2). Conversely, assume that the $\ug_k$'s are in $\int_{\Aa_f(\SQ)}(\uE,\hZ)$. Because of the degrees, every $\ug\in\int_{\Aa_f(\SQ)}(\uE,\hZ)$ of degree $\leq n$ may be written $\ug(x)=\sum_{k=0}^n\uc_k\ug_k(x)$ with $\uc_k\in\Aa_f(\SQ)$. Since, for every $k$, $\ug_k(\ua_h)=\underline{0}$ for $0\leq h\leq k-1$ and $\ug_k(\ua_k)=\underline{1}$, it is easy to see by induction on $k$ that $\ug(\ua_h)\in\hZ$ for $0\leq h\leq k$ implies that the coefficients $\uc_0,\ldots,\uc_k$ belong to $\hZ$. Indeed, the coefficients $\uc_k$ may be computed recursively by:
\be\label{eq:46} \uc_k=\ug(\ua_k)-\sum_{h=0}^{k-1}\uc_h\ug_h(\ua_k)\,.\ee
Consequently, the $\ug_k$'s form a $\SZ$-basis, so that (2)$\rightarrow$(3).

Finally that (3)$\leftrightarrow$(4) follows easily from the previous proof of (2)$\rightarrow$(3).
\end{proof}

\begin{example}
The sequence $\{\underline{n}\}_{n\geq 0}$ (where $\underline{n}=(n_p)_{p\in\SP}$ and $n_p=n$ for every $p$) is an adelic ordering of $\hZ$. Thus, the polynomials $\binom{x}{n}$ form a regular basis of $\int_{\Aa_f(\SQ)}(\hZ):$ every $\ug\in\int_{\Aa_f(\SQ)}(\hZ)$ of degree $n$ may be uniquely written as $\sum_{k=0}^n\uc_k\binom{x}{k}$ where the $\uc_k$'s are in $\hZ$ and satisfy:
$$\uc_k=g(\underline{k})-\sum_{h=0}^{k-1}\uc_h\binom{k}{h}\in\hZ\,.$$
\end{example}
\section{Extension of Bhargava-Kedlaya's theorem}

Recall that, in $\hZ$, a sequence $\{\ua_n\}_{n\geq 0}=\{(\alpha_{n,p})_{p\in\SP}\}_{n\geq 0}$ converges to $\ua=(\alpha_p)_{p\in\SP}$  if and only if, for each $p\in\SP$, the sequence $\{\alpha_{n,p}\}_{n\geq 0}$ converges to $\alpha_p$ in $\SZ_p$.  
And, in $\Cc(\uE,\hZ)$, a sequence $\{\ufi_n\}_{n\geq 0}=\{(\varphi_{n,p})_{p\in\SP}\}_{n\geq 0}\in \prod_{p\in\SP}\Cc(E_p,\SZ_p)$ converges uniformly to $\ufi=(\varphi_p)_{p\in\SP}$ if and only if, for each $p\in\SP$, $\{\varphi_{n,p}\}_{n\geq 0}$ converges uniformly to $\varphi_p$ in $\Cc(E_p,\SZ_p)$.

It follows from the previous section that Mahler's result (Theorem~\ref{th:mah}) extends in the following way:

\begin{proposition}\label{th:mal2}
Every $\ufi\in\prod_{p\in\SP}\Cc(\SZ_p,\SZ_p)$ may be uniquely written as a series in the $\binom{x}{n}$'s with coefficients in $\hZ:$
$$\ufi(x)=\sum_{n=0}^\infty \uc_n\binom{x}{n}, \textrm{ where } \uc_n\in\hZ,\; \lim_{n\to+\infty}\uc_n=\underline{0}\,.$$
Moreover,
$$\uc_k=\ufi(\underline{k})-\sum_{h=0}^{k-1}\uc_h\,\binom{k}{h}\in\hZ\,.$$
\end{proposition}

More generally and more precisely, we extend Bhargava-Kedlaya's result (Theorem~\ref{th:BK}):

\begin{theorem}\label{th:BK1}
Assume that $\uE=\prod_{p\in\SP}E_p$ is a compact subset of $\hZ$ such that, for each $p\in\SP$, $E_p$ is infinite and, for each $n\geq 0$, $\#(E\pmod{p})>n$ for almost all $p\in\SP$. Then:
\begin{itemize}
\item[1)] There exists a sequence $\{\ua_n\}_{n\geq 0}$ which is an adelic ordering of $\uE$, and the corresponding polynomials 
$$\ug_n(x)=\prod_{k=0}^{n-1}\frac{x-\ua_k}{\ua_n-\ua_k}$$
form a basis of $\hZ$-module $\int_{\Aa_f(\SQ)}(\uE,\hZ)$.

\item[2)] Every $\ufi\in\prod_{p\in\SP}\Cc(E_p,\SZ_p)$ may be uniquely written as a uniformly convergent series in the $\ug_n$'s with coefficients in $\hZ:$
$$\ufi(x)=\sum_{n\geq 0}\uc_n\,\ug_n(x)Ê\textrm{ where } \uc_n\in\hZ \textrm { and } \lim_{n\to+\infty}\uc_n=\underline{0}\,.$$

\item[3)] The $\uc_n$'s satisfy the recursive formula:
\be\label{eq:52} \uc_n=\ufi(\ua_n)-\sum_{k=0}^{n-1}\uc_k\,\ug_k(\ua_n)\,,\ee
and one has:
$$\forall p\in\SP,\quad \inf_{n\geq 0}v_p(\pi_p(\uc_n)) =\inf_{\uy\in \uE}v_p(\pi_p(\ufi(\uy))\,.$$
\end{itemize}
\end{theorem}

\begin{proof}
The first assertion follows from Corollary~\ref{th:mod}, Theorem~\ref{th:52}, Proposition \ref{th:cyclic characteristic ideals} and Proposition~\ref{th:56}.
The second assertion is then a straightforward consequence of Theorem~\ref{th:BK} since $\ufi=(\varphi_p)\in\prod_p\Cc(E_p,\SZ_p)$ and, if the $\ug_n$'s form a basis of the $\hZ$-module $\int_{\Aa_f(\SQ)}(\uE,\hZ)$ then, for every $p\in\SP,$ the $\pi_p(\ug_n)$'s form a basis of the $\SZ_p$-module $\int_{\SQ_p}(E_p,\SZ_p)$, by definition of adelic ordering.
The last assertions are consequences of formula (\ref{eq:46}),  Theorem~\ref{th:BK} and the fact that
since $\{\ua_n\}_{n\geq 0}$ is an adelic ordering of $\uE$ then, for each $p\in \SP$, $\{\pi_p(\ua_n)\}_{n\geq 0}$ is a $p$-ordering of $E_p=\pi_p(\uE)$. 
\end{proof}

In fact, we can write the functions $\ufi$ as sums of series of polynomials even if these polynomials are not associated to an adelic ordering of $\uE$. It is enough that they form a basis of the $\hZ$-module $\int_{\Aa_f(\SQ)}(\uE,\hZ)$. Moreover, thanks to Lemma~\ref{rem:1}, we may consider functions $\ufi$ defined on a compact subset of $\Aa_f(\SQ)$ and with values in $\Aa_f(\SQ)$.

\begin{theorem}\label{th:63}
Assume that $\uE=\prod_{p\in\SP}E_p$ is a compact subset of $\Aa_f(\SQ)$ such that, for each $p\in\SP$, $E_p$ is infinite and, for each $n\geq 0$, $\#(E\pmod{p})>n$ for almost all $p\in\SP$. Then, there exist regular bases, in particular regular bases formed by polynomials in $\SQ[x]$, of the $\hZ$-module $\int_{\Aa_f(\SQ)}(\uE,\hZ)$. And, for any regular basis $\{\ug_n(x)\}_{n\geq 0}$, every $\ufi\in \Cc(\uE,\Aa_f(\SQ))\cap\prod_{p\in\SP}\Cc(E_p,\SQ_p)$ may be uniquely written as a uniformly convergent series in the $\ug_n$'s with coefficients in $\Aa_f(\SQ):$
$$\ufi(x)=\sum_{n\geq 0}\uc_n\,\ug_n(x)Ê\textrm{ where } \uc_n\in\Aa_f(\SQ) \textrm { and } \lim_{n\to+\infty}\uc_n=\underline{0}\,.$$
Moreover, 
$$\forall p\in\SP\quad \inf_{n\geq 0}v_p(\pi_p(\uc_n)) =\inf_{\uy\in \uE}v_p(\pi_p(\ufi(\uy))\,.$$
\end{theorem}

The existence of regular bases, in particular bases formed by polynomials in $\SQ[x]$, follows from Theorem~\ref{th:52} and Corollary~\ref{th:mod}. But, we have to take care that, in general, the coefficients $\uc_n$ do not satisfy the recursive formula~(\ref{eq:52}) which implied easily the uniqueness of the coefficients in the previous theorem. Nevertheless, the existence and the uniqueness of the coefficients are consequences of an extension of  Theorem~\ref{th:BK}: following \cite[Theorem 2]{bib:Bharg2} or \cite[Theorem 2.7]{bib:CC2}, the bases associated to orderings may be replaced by any basis. 

Be careful to also note that to say that $\lim_{n\to+\infty}\uc_n=\underline{0}$ in $\Aa_f(\SQ)$ means not only that, for every $p\in\SP$, $\lim_n\pi_p(\uc_n)=0$ in $\SQ_p$, but also that there exist $N$ such that $\uc_n\in\hZ$ for $n\geq N$. The fact that the series $\sum_n\uc_n\,\ug_n(x)$ converges uniformly in $\Cc(\uE,\hZ)$ is an obvious consequence of both conditions satisfied by the $\uc_n$'s.

\begin{remark}
In Proposition~\ref{th:mal2}, the functions of $\prod_{p\in\SP}\Cc(\SZ_p,\SZ_p)$ were expanded as series in the binomial polynomials $\binom{x}{n}$. Note that these $\binom{x}{n}$'s form a regular basis of the $\hZ$-module $\int_{\Aa_f(\SQ)}(\hZ,\hZ)$ which corresponds either to an adelic ordering of $\hZ$ as in Theorem~\ref{th:BK1}, or to polynomials in $\int_{\SQ}(\hZ,\hZ)$ as in Theorem~\ref{th:63}.
\end{remark}

\subsection*{Acknowledgments}

The authors wish to thank the referee for carefully reading the paper.

\end{document}